\documentclass{article}
\usepackage[utf8]{inputenc}
\usepackage{amsfonts}
\usepackage{amssymb}
\usepackage{tikz-cd}
\usepackage{mathtools}
\usepackage{amsmath}
\usepackage{graphicx}
\usepackage{commath}
\graphicspath{images/}
\usepackage[english]{babel}
\usepackage{dirtytalk}
\usepackage{amsthm}
\usepackage{cancel}
\usepackage{enumerate}
\usepackage{authblk}
\usepackage{hyperref}
\usepackage{verbatim}
\usepackage{url}
\usepackage{bbold}
\usepackage{nicematrix}

\usepackage[maxbibnames=99]{biblatex}
\newtheorem{Theorem}{Theorem}[section]

\newtheorem{lemma}[Theorem]{Lemma}
\newtheorem{proposition}[Theorem]{Proposition}
\newtheorem{corollary}[Theorem]{Corollary}

\newtheorem{remark}[Theorem]{Remark}
\newtheorem*{prob}{Problem J.7}
\newtheorem{defi}[Theorem]{Definition}
\newtheorem*{thrma}{Theorem A}
\newtheorem*{thrmb}{Theorem B}
\newtheorem*{thrmc}{Theorem C}
\newtheorem*{exam}{Examples}

\newcommand{\bigzero}{\mbox{\normalfont\Large\bfseries 0}}

\newcommand{\rvline}{\hspace*{-\arraycolsep}\vline\hspace*{-\arraycolsep}}

\def\C{\mathbb{C}}
\def\R{\mathbb{R}}
\def\Z{\mathbb{Z}}
\def\D{\mathbb{D}}

\makeatletter
\newcommand{\tpitchfork}{%
  \vbox{
    \baselineskip\z@skip
    \lineskip-.52ex
    \lineskiplimit\maxdimen
    \m@th
    \ialign{##\crcr\hidewidth\smash{$-$}\hidewidth\crcr$\pitchfork$\crcr}
  }%
}
\makeatother

\makeatletter
\newcommand{\address}[1]{\gdef\@address{#1}}
\newcommand{\email}[1]{\gdef\@email{\url{#1}}}
\newcommand{\@endstuff}{\par\vspace{\baselineskip}\noindent\small
\begin{tabular}{@{}l}\scshape\@address\\\textit{E-mail address:} \@email\end{tabular}}
\AtEndDocument{\@endstuff}
\makeatother

\title{Embeddings and disjunction of Lagrangian pinwheels via rational blow-ups}

\author{Nikolas Adaloglou}
\date{}

\address{Mathematical Institute, Leiden University, The Netherlands}
\email{n.adaloglou@math.leidenuniv.nl}

\begin{document}

\maketitle

\begin{abstract}
We use the symplectic rational blow-up to study some Lagrangian pinwheels in symplectic rational manifolds. In particular, we determine which symplectic forms in the threefold blow-up of $\C P^2$ carry Lagrangian projective planes that can be made disjoint by a Hamiltonian isotopy. In addition, we show that such a disjunction is not possible in del Pezzo surfaces with Euler characteristic between $4$ and $7$. Finally, we determine which symplectic forms on $S^2\times S^2$ carry a Lagrangian $L_{3,1}$ pinwheel, answering a question of J. Evans.
\end{abstract}

\section*{Introduction}
Lagrangian pinwheels are certain simple two-dimensional CW-complexes in symplectic $4$-manifolds. They come up as vanishing cycles of surface quotient singularities, generalizing the more usual situation where Lagrangian spheres can be viewed as vanishing cycles of nodal singularities. In this note, we are going to study the behaviour of the Lagrangian pinwheels  $L_{2,1}$, which is just a usual Lagrangian $\R P^2$, and $L_{3,1}$ in some rational symplectic $4$-manifolds.\par 

Various rigidity results have been obtained for Lagrangian pinwheels. For example, Evans-Smith showed in \cite{ES} a rich intersection pattern for Lagrangian pinwheels in $\C P^2$, mimicking the algebro-geometric pattern for $\mathbb{Q}$-Gorenstein Deformation of $\C P^2$. Specifying to the case that the pinwheel is $\R P^2$, Shevchishin-Smirnov, in \cite{SS}, proved a non-squeezing result for Lagrangian $\R P^2$'s in $X_3$, the triple blow-up of $\C P^2$, completely determining which symplectic forms on $X_3$ can carry a Lagrangian $\R P^2$.\par 

The main technique of \cite{SS} is a symplectic surgery known as the rational blow-up. The rational blow-up removes a neighborhood of the Lagrangian pinwheel and glues back a certain configuration of symplectic spheres. A key feature of this surgery is that it preserves the rationality of symplectic manifolds, as was proved by Park-Shin in \cite{ParkShin}, making it possible to determine the symplectic manifold resulting after the surgery.\par

Our goal with this paper is to show the surprising usefulness of the symplectic rational blow-up on providing obstructions to certain behaviours of Lagrangian pinwheels. Our first such result is to determine exactly when a Hamiltonian isotopy can make two Lagrangian projective planes disjoint. We consider $\{H,E_1,E_2,E_3\}$ the standard basis for $H_2(X_3,\Z)$ and $h,\mu_1,\mu_2,\mu_3$ the corresponding periods (see Section \ref{basicsympgeometry} for more details).  We then have:

\begin{thrma}\emph{(=Theorem \ref{tworp2thrm})}
Let $L_1$ and $L_2$ be Lagrangian $\R P^2$'s in $(X_3,\omega)$. They can be made disjoint with a Hamiltonian isotopy if and only if the periods $h,\mu_1,\mu_2,\mu_3$ of $\omega$ satisfy the inequality

\[\mu_1+\mu_2+\mu_3<h.\]
\end{thrma}

A well studied class of rational symplectic manifolds are the symplectic del Pezzo surfaces $\D_n$. These are rational symplectic manifolds which are also monotone. Theorem $A$ shows that any two Lagrangian projective spaces must intersect in $\D_3$. We then generalize this to show

\begin{thrmb}[= Theorem \ref{delpezzothrm}]
In $\D_n$, for $n\leq 6$, any two Lagrangian $\R P^2$'s must intersect.
\end{thrmb}

The last theorem we are going to prove concerns the pinwheel $L_{3,1}$, which answers a question of Evans in \cite{evans_2023}.

\begin{thrmc}[=Theorem \ref{jquestion}]
In $(S^2\times S^2,\omega_\lambda)$, the $\Z_3$-homology class $[A+B]$ carries a Lagrangian $L_{3,1}$ pinwheel if and only if
\[\frac{1}{2}<\lambda<2.\]
\end{thrmc}

The paper is organized as follows: First, we recall basic properties Lagrangian pinwheels, the symplectic rational blow-up/down and the basic symplectic geometry of symplectic rational surface. In the second section we prove Theorem $A$ and then use induction to prove Theorem $B$. In the last section, Theorem $C$ is proved.

\section*{Aknowledgements}
The author is grateful to: George Politopoulos for his constant help with the algebro-geometric notions and everything else; Johannes Hauber for many useful and interesting discussions; Jesse Vogel for assisting computations with his Python knowledge; Felix Schlenk for providing a hospitable and stimulating environment in Neuch\^{a}tel in October 2022; Jonny Evans for his interest in this work and for suggesting to look at Problem J.$7$ from his new book. Finally, the author is also grateful to Federica Pasquotto, his PhD advisor, for her overall guidance and support. This paper was written as part of the author's PhD thesis in Leiden University.

\section{Recollections}

\subsection{Lagrangian pinwheels and the Symplectic Rational Blow-up/down}\label{ratblowupsec}
Let us recall the definitions around Lagrangian pinwheels, following \cite{BS}(see also \cite{khodorovskiy2013symplectic},\cite{ES}). 

\begin{defi}
An $n$-pinwheel, for $n\in \mathbb{N}_{\geq 2}$, is the topological space obtained by gluing a $2$-cell to a circle along the map $z\rightarrow z^{n}$.
\end{defi} 

Pinwheels where first studied by Tatyana Khodorovskiy in \cite{khodorovskiy2013symplectic}, where she showed that they can be considered as the Lagrangian skeleton to certain Stein manifolds. These are the $\mathbb{Q}$-homology balls $B_{n,1}$ and have as boundary the Lens space $L(n^2,n-1)$. Even though it is somewhat technical to formalize what it means for a CW-complex to be a Lagrangian subspace, based on Khodorovskiy's work we can give the following succinct definition:

\begin{defi}
In a symplectic manifold $(M^4,\omega)$, a Lagrangian pinwheel $L_{n,1}$ is an $n$-pinwheel that has an open neighborhood symplectomorphic to an open neighborhood of the Lagrangian skeleton of $B_{n,1}$.
\end{defi}

When $n=2$ the corresponding $L_{2,1}$ Lagrangian Pinwheel is just a Lagrangian $\R P^2$ and  $B_{2,1}$ is symplectomorphic to the unit cotangent disc bundle of $\R P^2$. \par

The $B_{n,1}$ play an important role in Algebraic Geometry as they appear as Milnor Fibers of some quotient singularities. In this context, the core pinwheel is the vanishing cycle of the smoothing. In the symplectic topology of $4$-manifolds, the come up in the context of the symplectic rational blow-up and down. We refer the reader to Chapters 7-9 of \cite{evans_2023} for a more thorough discussion on the algebro-geometric aspects of pinwheels.\par

The symplectic rational blow-down, first studied by Margaret Symington in \cite{Sym}, is a surgery done around a configuration $\mathcal{S}_n$ of transversally intersecting symplectic spheres. One removes a neighborhood of $\mathcal{S}_n$ and glues back the symplectic rational homology ball $B_n$, which has as core the Lagrangian pinwheel $L_{n,1}$.  \par

Interestingly the symplectic rational blow-down preceded the inverse operation, the symplectic rational blow-up, which was introduced in \cite{khodorovskiy2013symplectic}. As the inverse of the symplectic rational blow-down, the symplectic rational blow-up replaces the a symplectically embedded $B_{n,1}$ with the configuration $\mathcal{S}_n$.\par

Let us note here that blowing-up an $L_{2,1}$, i.e. a Lagrangian $\R P^2$ replaces it with $(-4)$ symplectic sphere. Analogously, blowing-up an $L_{3,1}$ replaces it with symplectic $(-5)$ and $(-2)$ spheres, having a unique intersection which is symplectically orthogonal. 

Both operations are symplectic surgeries in the sense that when performed on a symplectic manifold $(M,\omega)$ they produce a new $4$-manifold $\widetilde{M}$ which also carries a symplectic form $\tilde{\omega}$. Focusing on the symplectic rational blow-up, since the surgery takes place around a small neighborhood of an embedded $L_n$, one has a symplectomorphism between $(X-\nu L_n,\omega)$ and $(\widetilde{X}-\nu\mathcal{S}_n,\tilde{\omega})$, where $\nu$ denotes a small tubular neighborhood of the relevant subspaces. In addition, a simple Mayer-Vietoris argument shows that if $\widetilde{X}$ is the rational blow-up of $X$ along an $L_{n}$ then $b_2(\widetilde{X})=b_2(X)+n$.\par

\subsection{Basic Symplectic Geometry of rational surfaces}\label{basicsympgeometry}

The symplectic $4$-manifolds that we will study are the symplectic rational surfaces. The underlying smooth manifold is either $S^2\times S^2$ or the $k$-th point blow-up of $\C P^2$, denoted by $X_k$. We will denote the standard basis of $H_2(X_k,\Z)$ with $H,E_i$ for $1\leq i\leq k$.\par 

Using Seiberg-Witten theory one can establish the following facts for symplectic rational surfaces that make them very simple to work with:

\begin{Theorem}\emph{(\cite{li2008space})}\label{basicsymptoprational}
Let $X$ be a smooth, oriented rational manifold and $\Omega\in H^2(X,\R)$. Suppose that $\Omega^2([X])>0$.  Then
\begin{itemize}
    \item $\Omega$ can be represented by a symplectic form $\omega$, compatible with the orientation, if and only if $\Omega(E)>0$ for all $E\in H_2(X,\Z)$ where $E$ can be represented by a smooth $(-1)$-sphere. In fact, any symplectic form $\omega$ carries a symplectic $(-1)$-spheres in any homology class that caries smooth $(-1)$-spheres.
    
    \item Any two symplectic forms representing $\Omega$ are diffeomorphic.
\end{itemize}

\end{Theorem}

\begin{remark}\label{basicremark}
Following the second statement of the above theorem we will often identify a symplectic form $\omega$ with its cohomology class $\Omega=[\omega]$. Such cohomology class is determined by its periods. For $X_k$, we follow \cite{SS}  and denote the periods by 

\[h=\int_H\omega,\quad \mu_i=\int_{E_i}\omega.\]
\end{remark}

The last property of symplectic rational manifolds we want to use is that they are, as a collection of symplectic manifolds, closed under the symplectic rational blow-up. This was first proved in the $L_{2,1}$ case in \cite{BLW} and in the general $L_{n,1}$ case by Park-Shin in \cite{ParkShin}.

\begin{Theorem}\emph{(Closedness under rational blow-up)}\label{ratblowuponratman}
Let $(X,\omega)$ be a rational symplectic manifold and $L$ a Lagrangian pinwheel. After rationally blowing-up $X$ along $L$, the resulting symplectic manifold $(\widetilde{X},\tilde{\omega})$ is again rational.
\end{Theorem}

\begin{remark}
This result is certainly false for the case of the symplectic rational blow-down. Remark 3.1 in \cite{BLW} explains an example found in \cite{Dorfm}, where an explicit symplectic rational blow-down in performed on a symplectic rational manifold and the resulting symplectic manifold is \textbf{not} rational.
\end{remark}

The usefulness of Theorem \ref{ratblowuponratman} comes from the fact that, given a symplectic rational manifold $(X,\omega)$ and a Lagrangian pinwheel $L_n$, we can completely determine the diffeomorphism type of the blown-up manifold $\widetilde{X}$ since, up to diffeomorphism, rational manifolds are determined by their second Betti number. The only exceptions to this are the rational manifolds $S^2\times S^2$ and $X_1$ but in the cases we examine, that ambiguity is easily resolved.

Here are some examples that will show up in the rest of the paper.

\begin{exam}

\begin{enumerate}
\hfill

\item The most basic rational surface is $(\mathbb{C}P^2,\omega_{FS})$. Since $H_2(\mathbb{C}P^2,\mathbb{Z})=\mathbb{Z}$, we get that any symplectic form on $\mathbb{C}P^2$ is diffeomorphic to a scalar multiple of the Fubini-Study form. Here, given an orientation, the only inequality the period of a cohomology class needs to satisfy to carry a symplectic form is $h>0$.
   
\item For $X_k$ with $0\leq k\leq 8$, we can completely determine the homology classes that carry exceptional spheres. In the standard basis, these classes are
         
         \begin{enumerate}
            \item $E_i-E_j$,
            \item $H-E_i-E_j$,
            \item $2H-\sum_{i=1}^5 E_i$, where the $E_i$ are $5$ exceptional spheres, part of a basis,
            \item $3H-2E_j-\sum_{i=1}^6 E_i$ where the $E_i$ and $E_j$ are $7$ exceptional spheres, part of a basis.
        \end{enumerate}
    Since these are a finite number of classes, we can completely determine the inequalities that the periods of a cohomology class must satisfy in order to carry a symplectic form. For example, as was spelled out in \cite{SS}, a cohomology class of $X_3$ carries a symplectic form if and only, apart from having positive total volume, its periods satisfy $h,\mu_i>0$ and $h-\mu_i-\mu_j>0$ for $i,j=1,2,3$.

 \item The rational manifolds $X_k$ with $0\leq k \leq 8$ have (up to scaling or symplectomorphism) a distinguished symplectic form that makes them monotone. The periods of these forms satisfy $\mu_i=\frac{h}{3}$. Additionally, they are compatible with the complex structure that makes the $X_k$ a del Pezzo surface, thus their significance. We will denote such a symplectic rational manifold with $\D_k$, often omitting the symplectic form.
   
 \item Much related to $X_1$ is another simple symplectic manifold $(S^2\times S^2,\omega_\lambda)$ where $\omega_\lambda$ is the symplectic form which has area $\lambda$ on $A=[S^2\times \star]$ and area $1$ on $B=[\star\times S^2]$. $X_1$ and $S^2\times S^2$ have the same cohomology groups but with different ring structures, since $A^2=0=B^2$ and $A\cdot B=1$. It should be noted that rationally blowing-up a Lagrangian $\mathbb{R}P^2$ in $(\mathbb{C}P^2,\omega_{FS})$ transforms it to $(S^2\times S^2,\omega_\lambda)$ where the Lagrangian $\mathbb{R}P^2$ is replaced by a $(-4)$-symplectic sphere.

\end{enumerate}   
\end{exam}

Following Example 4 above, we want to determine what happens when we blow-up two disjoint Lagrangian $\R P^2$ in $X_3$, as we will need this information later. This calculation is slightly more elaborate so we prove the following lemma in some detail:

\begin{lemma}\label{existbasislemma}
Let $L_1$ and $L_2$ be two disjoint Lagrangian $\R P^2$'s in $(X_3,\omega)$, lying the homology classes $H$ and $E_1+E_2+E_3$ respectively. Performing symplectic rational blow-up on both of them results to the symplectic rational manifold $(X_5,\tilde{\omega})$. In addition, there exists a basis $\{Z_\infty,F,\widetilde{E}_0,\widetilde{E}_1,\widetilde{E}_2,\widetilde{E_3}\}$ for $H_2(X_5,\Z)$ where the intersection form becomes 

\[
\begin{pmatrix}
  \begin{matrix}
  -4 & 1 \\
  1 & 0
  \end{matrix}
  & \rvline & \bigzero \\
\hline
  \bigzero & \rvline &
  \begin{matrix}
  -1 & 0 & 0 & 0 \\
  0 & -1 & 0 & 0 \\
  0 & 0 & -1 & 0 \\
  0 & 0 & 0 & -1
  \end{matrix}
\end{pmatrix}
.\]

Finally, the symplectic $(-4)$-spheres $S_1$ and $S_2$, replacing the $L_i$, represent, up to diffeomorphism, the homology classes $Z_\infty$ and $\widetilde{E}_0-\widetilde{E}_1-\widetilde{E}_2-\widetilde{E}_3$.
\end{lemma}

\begin{proof}
Blowing-up the $L_i$ in $X_3$ yields a rational manifold with $b_2=6$ and, up to diffeomorphism, this manifold is $X_5$.\par

Let us construct the claimed basis; consider $S^2\times S^2$ as the underlying smooth manifold of the Hirzebruch surface with a $(-4)$-rational curve. Let $Z_\infty$ be this curve and $F$ one of fibers of the $P^1$-fibration. Consider the complex manifold obtained by blowing-up $4$ distinct points away from $Z_\infty$ and let the $\widetilde{E}_i$, for $0\leq i\leq 3$ be the exceptional divisors. The underlying smooth manifold is again $X_5$ and the rational curves $Z_\infty,F,\widetilde{E}_i$ provide the desired basis.\par 

We are left with identifying the homology classes of $S_1$ and $S_2$. Proposition 3.2 in \cite{BLW} shows that diffeomorphisms act transitively on the set representing such homology classes. Thus we may assume that $S_1$ is indeed in $Z_\infty$ and then it is straightforward to check that the only other possible class for $S_2$ is $\widetilde{E}_0-\widetilde{E}_1-\widetilde{E}_2-\widetilde{E}_3$ (up to permuting the $\widetilde{E}_i)$.

m
\end{proof}

\subsection{Homology classes of Lagrangian $\R P^2$'s.}
We establish here which homology classes can be represented by Lagrangian $\R P^2$'s. We will make heavy use of the fact that any smoothly embedded $(-2)$-sphere $S$ defines, up to smooth isotopy, a diffeomorphism which we call a \textit{Dehn twist along $S$}. Dehn twists along a sphere generalize the usual notion of a Dehn twist along a circle.\par

The key property of a Dehn twist $\phi$ along $S$ is that it acts on homology as a reflection along $[S]$, meaning that

\begin{equation}\label{homdehn}
 \phi_*(A)=A+(A\cdot [S])[S].   
\end{equation}

For more information on Dehn twists we refer the reader to \cite{BLW}. We note here that if $S$ is Lagrangian, the Dehn twist can be assumed to be a symplectomorphism of the ambient manifold.\par 

\begin{lemma}\label{existence1rp2}
Let $X_k$ be a rational manifold with $0\leq k\leq 8$. Then, the $\Z_2$-homology classes that carry a Lagrangian $\R P^2$ for \textit{some} symplectic form are exactly
\begin{enumerate}[i)]
    \item $H$ for $k\geq 0$,
    \item $E_i+E_j+E_k$ for $k\geq 3$,
    \item $H+E_i+E_j+E_k+E_l$ for $k\geq 4$,
    \item $\sum_{i=1}^7 E_i$ where $E_i$ are $7$ disjoint exceptional spheres belonging to a maximal collection, for $k\geq 7$,
    \item $H+\sum_{i=1}^{8} E_i$ where $E_i$ belong to a maximal collection of exceptional spheres, for $k=8$.
\end{enumerate}
\end{lemma}

The proof of the above lemma can be inferred from Proposition \ref{existence1rp2} in \cite{DHL}. For the convenience of the reader we also provide a more straightforward proof. 

\begin{proof}
The standard cohomological constraint proven by Audin in \cite{AudinLag} shows that a $\Z_2$-homology class of $X_k$ can carry a Lagrangian $\mathbb{R}P^2$ if and only if it a $\Z_2$ reduction of a $\Z$-homology class with self intersection $1\mod 4$. It is straightforward to check that only the above classes satisfy this constraint.\par 
The fact that these classes indeed carry a Lagrangian $\mathbb{R}P^2$ for some symplectic form essentially follows from the fact that for each of these classes there exists a diffeomorphism of $X_k$ mapping $H$ to the corresponding class. Indeed:
\begin{enumerate}[i)]
    \item $\C P^2$ carries a Lagrangian $\mathbb{R}P^2$, say $L$, in the class $H$ for all symplectic forms $\omega$. For sufficiently small radii, one can symplectically embed $k$ balls in the complement of $L$ and then blow them up. Therefore $L$ exists also in $(X_k,\tilde{\omega})$, where $\tilde{\omega}$ is the appropriate symplectic form on the blow-up.
    
    \item $H$ can be mapped to $E_i+E_j+E_k$ via a Dehn twist along a $(-2)$-sphere in the class $H-E_i-E_j-E_k$.
   
    \item Assume $k\geq 5$. The class $H+E_i+E_j+E_k+E_l$ is mapped to $E_k+E_l+E_{l'}$ where $l'\neq k,l$ via a Dehn twist along a $(-2)$ sphere in $H-E_i-E_j-E_{l'}$ and this further can be mapped to $H$. For $k=4$, consider a Lagrangian $\mathbb{R}P^2$ in the class $H+E_1+E_2+E_3+E_4$ in $(X_5,\omega)$. By Corollary 3.3 in \cite{BLW}, there exists a symplectic exceptional sphere in $E_5$, disjoint from the Lagrangian $\R P^2$. Blowing down this sphere gives a Lagrangian $\mathbb{R}P^2$ in $X_4$ in the class $H+E_1+E_2+E_3+E_4$.
    
    \item Consider three different summands, say $E_1,E_2,E_3$, of $\sum_{i=1}^7 E_i$. Dehn twisting along a $(-2)$-sphere in $H-E_1-E_2-E_3$ maps the class $\sum_{i=1}^7 E_i$ to $H+E_4+E_5+E_6+E_7$ and this can be further mapped to $H$.

    \item Consider $H+\sum_{i=1}^8 E_i$ in $X_9$. Dehn twisting along a $(-2)$-sphere in $H-E_1-E_2-E_{9}$ maps the class $H+\sum_{i=1}^8 E_i$ to $E_3+E_4+\cdots + E_9$ which can be mapped to $H$. Therefore there exists some $\tilde{\omega}$ such that $(X_9,\tilde{\omega})$ carries a Lagrangian $\R P^2$ in $H+\sum_{i=1}^8 E_i$. By Corollary 3.3 in \cite{BLW} one can find an exceptional sphere in $E_9$ disjoint from the Lagrangian $\R P^2$ and blowing down that sphere produced the desired Lagrangian also for $k=8$.
\end{enumerate}
\end{proof}

It is straightforward to also distinguish pairs of homology classes carrying homologically disjoint Lagrangian $\R P^2$'s. 

\begin{lemma}\label{existence2rp2}
Let $L_1,L_2$ be two disjoint Lagrangians in $(X_k,\omega)$ for $3\leq k\leq 8$. Up to diffeomorphism (and swapping the two Lagrangians),  $L_1$ and $L_2$ belong to the classes
\begin{itemize}
    \item $H$ and $E_1+E_2+E_3$,
    \item $H$ and $E_1+E_2+E_3+E_4+E_5+E_6+E_7$.
\end{itemize}
\end{lemma}

\begin{proof}
The classes $H+E_1+E_2+E_3+E_4$ in $X_4$ and  $H+\sum_{i=1}^8 E_i$ in $X_8$ pair non-trivially with any other class possibly carrying a Lagrangian $\R P^2$ so we will not consider them.\par 

In the proof of Lemma \ref{existence1rp2} we showed that any other class can  be mapped to $H$ via a composition of Cremona transforms. Thus we can assume that $L_1$, after a diffeomorphism, represents the class $H$. Then, the only options for $L_2$, up to a diffeomorphism fixing $H$,  are $E_1+E_2+E_3$ and $E_1+E_2+E_3+E_4+E_5+E_6+E_7$.
\end{proof}

\section{Lagrangian projective planes in symplectic rational manifolds:}

\subsection{In $X_3$.}
In \cite{SS}, Shevchishin-Smirnov determine completely which symplectic forms on $X_3$ carry a Lagrangian $\mathbb{R}P^2$ in the homology class $E_1+E_2+E_3\in H_2(X_3,\mathbb{Z}_2)$. We will give a sketch for the proof of their theorem, since our proof will follow theirs in logic.

\begin{Theorem}\emph{(\cite{SS})}\label{thrmss}
The symplectic rational manifold $(X_3,\omega)$ carries a Lagrangian $\mathbb{R}P^2$ in the homology class $E_1+E_2+E_3$ if and only if 
\[\mu_i+\mu_j<\mu_k\]
\end{Theorem}

\begin{proof}\emph{(sketch)}
Let us start with the necessity of the inequality. By the rational blow-up/blow-down procedure, the existence of a Lagrangian $\mathbb{R}P^2$ in $(X_3,\omega)$ is equivalent to the existence of a symplectic $(-4)$-sphere $S$ in $(X_4,\widetilde{\omega})$.\par

Consider the standard basis $\{H,\widetilde{E}_0,\widetilde{E}_1,\widetilde{E}_2,\widetilde{E}_3\}$ of $H_2(X_4,\mathbb{Z})$ where the $\tilde{E}_i$ are represented by disjoint $(-1)$-spheres. Adjunction shows that $S$ represents the homology class $\widetilde{E}_0-(\widetilde{E}_1+\widetilde{E}_2+\widetilde{E}_3)$. Moreover, the (symplectic) identification of $X_3-\nu L$ with $X_4-\nu S$ carries over to the following homological identification:

\begin{equation}\label{basiceqss}
  E_i-E_j \leftrightarrow \widetilde{E}_i-\widetilde{E}_j,\quad 2E_i \leftrightarrow \widetilde{E}_0-\widetilde{E}_i+\widetilde{E}_j+\widetilde{E}_k.  
\end{equation}

Using Equation (\ref{basiceqss}), we can relate the periods $\tilde{\mu}_k$ of $(X_4,\tilde{\omega})$ to the periods $\mu_i$ of $(X_3,\omega)$. Indeed, the area positivity $S$ implies that for some $\epsilon>0$ we have $\tilde{\omega}(S)=4\epsilon$. Then the $\tilde{\mu}_k$ become

\begin{equation}\label{periodeqss}
   \tilde{\mu}_0=\frac{\mu_1+\mu_2+\mu_3}{2}+\epsilon,\quad \tilde{\mu}_k=\frac{\mu_i+\mu_j-\mu_k}{2}-\epsilon \text{ for } \{i,j,k\}=\{1,2,3\}. 
\end{equation}

Since $\tilde{\omega}$ is symplectic, Theorem \ref{basicsymptoprational} implies that $\tilde{\mu}_k>0$ and thus the $\mu_i$ satisfy the triangle inequality. \par

For the sufficiency part, we need to construct a Lagrangian $\mathbb{R}P^2$ under the assumption that the triangle inequality holds for the $\mu_i$. This is done by a Nakai-Moishezon argument. Consider the complex manifold structure on $X_4$ obtained by blowing-up a point in $\C P^2$ and then blowing-up three more distinct points of the exceptional divisor. In particular, this complex structure caries a complex $(-4)$ rational curve. One checks that the cohomology class $\Omega\in H_2(X_4,\R  $ with periods $\tilde{\mu}_i$ as in Equation \ref{periodeqss} satisfies the conditions of the Nakai-Moishezon criterion, thus the cohomology class $\Omega$ is represented by some K\"ahler form $\widetilde\omega$. Blowing down the $(-4)$ complex (also symplectic) sphere, gives a symplectic form $\omega$ on $X_3$ with periods exactly the number $\mu_i$ which further carries a Lagrangian $\R P^2.$ Thus, by Theorem \ref{basicsymptoprational} any symplectic form have the $\mu_i$ as periods carries such a Lagrangian $\R P^2$.

\end{proof}

\begin{remark}
For the sufficiency part of the above proof, Evans (\cite{EKb}) provides an alternative argument using visible Lagrangians.
\end{remark}

\begin{corollary}\label{corss}
Let the symplectic rational manifold $(X_3,\omega')$ have periods $h',\mu'_i$ with respect to the standard basis for homology. Then $(X_3,\omega')$ carries a Lagrangian $\mathbb{R}P^2$ in the homology class $H$ if and only if 

\[\mu'_k<\frac{h'}{2}\]

for $k=1,2,3$.

\end{corollary}

\begin{proof}
Let $\phi$ be a Dehn Twist along a $(-2)$-sphere in the homology class $H-E_1-E_2-E_3$ and write $\omega'=\phi^*(\phi_*\omega')$. By Equation \ref{homdehn}, the map $\phi_*$ maps the $\mathbb{Z}_2$-class $H$ to $E_1+E_2+E_3$. Therefore $(X_3,\omega')$ carries a Lagrangian $\mathbb{R}P^2$ in $H$ if and only if $(X_3,\phi_*\omega')$ carries a Lagrangian $\mathbb{R}P^2$ in $E_1+E_2+E_3$. 

Let $\mu_i',h'$ be the periods of $\omega'$ and $\mu_i,h$ the periods of $\phi_*\omega'$, corresponding to $E_i$ and $H$. Then

\[h=2h'-\mu_1'-\mu_2'-\mu_3',\quad \mu_k=h'-\mu'_i-\mu'_j\]

and we deduce that

\[\mu_i+\mu_j>\mu_k\Leftrightarrow \mu'_k<\frac{h'}{2}.\]

The proof concludes applying Theorem \ref{basiceqss} to $\phi_*\omega$. 

\end{proof}

\begin{remark}
It is interesting to notice that our proof of Corollary \ref{corss} shows that the Symplectic Triangle inequality of Shevchishin-Smirnov is essentially equivalent to the arguments of Borman-Li-Wu in Prop. 5.1 of \cite{BLW}. Nevertheless, the wonderful idea to examine embedding obstructions via directly blowing-up the Lagrangian $\R P^2$ is really a novelty of \cite{SS}.
\end{remark}

Having established necessary and sufficient conditions for the existence of Lagrangians in the classes $H$ and $E_1+E_2+E_3$, we now want to understand how the two Lagrangians, if they exist, intersect with each other. To do this, we are going to follow the same strategy as in Theorem \ref{basiceqss}, but with two disjoint Lagrangians instead of one.

For the sake of clarity we will split the proof of Theorem A in steps. Hopefully, this will also highlight the use of the rational blow-up.

We start by showing that as long as there exist \textit{some} disjoint Lagrangian $\R P^2$'s then \textit{any} two (in the right homology classes) can be made disjoint. The proof relies on a highly non-trivial result in \cite{LLWuniq}, concerning the transitivity of the action of Hamiltonian diffeomorphisms on homologous Lagrangian $\R P^2$'s.

\begin{lemma}\label{extoham}
Let $L_1$ and $L_2$ be two $\mathbb{R}P^2$'s in $(X_3,\omega)$. Then, $L_1$ and $L_2$ can be made disjoint by a Hamiltonian isotopy if and only if there exist two disjoint Lagrangians $L'_1$ and $L'_2$, homologous to $L_1$ and $L_2$.
\end{lemma} 

\begin{proof}
The \textit{only if} direction is clear. For the \textit{if} direction, suppose that such disjoint Lagrangians $K_i$'s exist, in the same homology classes as the $L_i$. By Corollary 1.5 of \cite{LLWuniq}, we see that Hamiltonian isotopies act transitively in Lagrangians $\mathbb{R}P^2$'s in the same homology class. By applying the same Corollary 1.5 twice, we get a Hamiltonian isotopy $\phi_t$ such that $\phi_1(L_1)=K_1$ and a Hamiltonian isotopy $\psi_t$ such that $\psi_1(\phi_1(L_2))=K_2$. The isotopy $\sigma_t=\phi_t^{-1}\circ \psi_t\circ \phi_t$ is again Hamiltonian and $\sigma_1(L_2)$ is disjoint from $L_1$.
\end{proof}

We now proceed to establish the necessary condition for the disjunction of the two Lagrangian projective planes. 

\begin{proposition}\emph{(Necessity)}\label{2rpnec}
Suppose that $(X_3,\omega)$ carries two disjoint Lagrangian $\R P^2$'s, $L_1$ and $L_2$. Then, the inequality
\[\mu_1+\mu_2+\mu_3<h\]
holds.
\end{proposition}

\begin{proof}
As per Lemma \ref{existence1rp2}, we may assume that $[L_1]=H$ and $[L_2]=E_1+E_2+E_3$. It is easy to see that the $\mathbb{Z}$-homology classes that pair trivially $\mod 2$ with $H$ and $E_1+E_2+E_3$ are exactly $E_i-E_j,2E_k$ and $2H$.\par

Let $(\widetilde{X},\tilde{\omega})$ be the rational symplectic manifold we obtain by blowing-up both $L_1$ and $L_2$. Since $b_2(\widetilde{X})=b_2(X_3)+2$, $\widetilde{X}$ is diffeomorphic to $X_5$.  For our computations, it is convenient to choose the special basis $\{ Z_\infty,F, \widetilde{E}_0,\widetilde{E}_1,\widetilde{E}_2,\widetilde{E}_3 \}$ for $H_2(\widetilde{X},\Z )$ as in Lemma \ref{existbasislemma}. In this basis, the symplectic spheres $S_1$ and $S_2$ replacing the $L_i$ may, and will, be assumed to represent the homology classes $Z_\infty$ and $\widetilde{E}_0-\widetilde{E}_1-\widetilde{E}_2-\widetilde{E}_3$. \par 

Denoting the periods of $\widetilde{\omega}$ by

\[\tilde{\omega}(\widetilde{E}_i)=\tilde{\mu}_i,\quad \tilde{\omega}(F)=\beta,\quad  \tilde{\omega}(Z_\infty+4F)=\alpha\]
we may write

\[PD(\tilde{\omega})=\alpha F +\beta Z_\infty-\sum \tilde{\mu}_i\widetilde{E}_i. \]
By the identification of $X_3-(L_1\cup L_2)$ with $\widetilde{X}-(S_1\cup S_2)$, comparing Chern classes and self intersections yields that the class $2H$ now corresponds to the class $Z_\infty+4F$ which implies

\begin{equation}
    \alpha=2h.
\end{equation}

In addition, the equalities between the $\mu_i$ and the $\tilde{\mu}_j$ hold as in the proof of Theorem \ref{thrmss}. In particular, for some $\epsilon>0$,

\[\tilde{\mu}_0=\frac{\mu_1+\mu_2+\mu_3}{2}+\epsilon.\]

Since $\tilde{\omega}$ is symplectic, it should evaluate positively on every exceptional class so in particular on $F-\widetilde{E}_0$:

\begin{equation}
    \tilde{\omega}(F-\widetilde{E}_0)>0\Leftrightarrow \beta>\tilde{\mu}_0
\end{equation}

 In addition, since $Z_\infty$ is symplectic we have 

\begin{equation}
    \tilde{\omega}(Z_\infty)=\alpha-4\beta>0 \Rightarrow \frac{h}{2}>\beta
\end{equation}

Putting the above together, we get the claimed inequality: 
\begin{align*}
    \frac{h}{2}>\beta&>\tilde{\mu}_0\overset{(\ref{periodeqss})}{=}\frac{\mu_1+\mu_2+\mu_3}{2}+\epsilon\Rightarrow \\
    h&>\mu_1+\mu_2+\mu_3.
\end{align*}

\end{proof}

We now move to the sufficiency part. The idea is to show that if the symplectic manifold $(X_3,\omega)$ satisfies the inequality $\sum \mu_i<h$, then there exists a symplectic (in fact, K\"ahler) form $\tilde{\omega}$ on $\tilde{X}$ such that $(X_3,\omega)$ is obtained by rationally blowing down two symplectic $(-4)$-spheres in
$(\widetilde{X},\tilde{\omega})$. \par 

The K\"ahler form $\tilde{\omega}$ is constructed via the Nakai-Moishezon Theorem, applied to the complex manifold $\widetilde{X}$. The complex structure of $\widetilde{X}$ comes from considering the Hirzeburch surface $F_4$ containing a $(-4)$-rational curve and blowing-up once, away from the $(-4)$ and $(+4)$ curves and then blowing-up three more times on different points at the exceptional divisor.\par

First, let us characterize $\mathcal{K}(\widetilde{X})$, the K\"ahler Cone of $\widetilde{X}$. Compare with Section $3.1.2$ of \cite{SS}.

 \begin{lemma}\label{kconex5}
A class $[\tilde{\omega}]\in H^{2}_{DR}(\widetilde{X})$ of the form

\[PD([\tilde{\omega}])=\alpha F+\beta Z_\infty-\sum \tilde{\mu}_i\widetilde{E}_i\]
lies in $\mathcal{K}(\widetilde{X})$ if and only if

\begin{itemize}
    \item $2\alpha\beta-4\beta^2-\sum^4_{i=0} \tilde{\mu}^2_i>0$,
    \item $\tilde{\mu}_i>0$ for $0\leq i\leq 3$,
    \item $\beta-\tilde{\mu}_0>0$,
    \item $\alpha>0$
    \item $\tilde{\mu}_0-\tilde{\mu}_1-\tilde{\mu}_2-\tilde{\mu}_3>0$,
    \item $\alpha-4\beta>0$.
\end{itemize}
\end{lemma}

\begin{proof}
We are going to use the fact that $[\tilde{\omega}]$ has a K\"ahler representative if and only if it has positive area on every irreducible curve. The first inequality is just total volume positivity, while the rest come from area positivity of certain curves: 
\begin{itemize}
    
    \item $\tilde{\omega}(\widetilde{E}_i)=\tilde{\mu}_i$ for $0\leq i\leq 3$,
    \item $\tilde{\omega}(F-\widetilde{E}_0)=\beta-\tilde{\mu}_0$,
    \item $\tilde{\omega}(Z_\infty+4F)=\alpha$,
    \item $\tilde{\omega}(\widetilde{E}_0-\widetilde{E}_1-\widetilde{E}_2-\widetilde{E}_3)=\tilde{\mu}_0-\tilde{\mu}_1-\tilde{\mu}_2-\tilde{\mu}_3$,
    \item $\tilde{\omega}(Z_\infty)=\alpha-4\beta$.
\end{itemize}

Now, we have to show that the inequalities in the statement of the lemma imply that any irreducible (not necessarily smooth) holomorphic curve $C$ in $(\widetilde{X},\tilde{\omega})$ has positive $\tilde{\omega}$-area. Since we have already shown that these inequalities imply that $\tilde{\omega}$ is positive on all of the classes above, we can assume that $C$ does not lie in any of these classes. Therefore, it has non-negative intersection with all of them. In homology, $C$ can be written as

\[C=\nu_fF+\nu_\infty Z_\infty-\sum\nu_i\widetilde{E}_i\]
and the non-negative intersection condition gives the following inequalities for the coefficients of $C$:

\begin{itemize}
    \item $C\cdot(\widetilde{E_0}-\widetilde{E_1}-\widetilde{E_2}-\widetilde{E_3})\geq 0\Rightarrow \nu_0\geq \nu_1+\nu_2+\nu_3 \quad (1b).$
    \item $C\cdot Z_\infty\geq 0\Rightarrow \nu_f\geq 4\nu_\infty\quad (2b).$
    \item $C\cdot F\geq 0\Rightarrow \nu_\infty\geq 0\quad (3b).$
    \item $C\cdot \widetilde{E}_i\geq 0\Rightarrow \nu_i\geq 0 \quad(4b).$
    \item $C\cdot (F-\widetilde{E}_0)\geq 0\Rightarrow \nu_\infty\geq \nu_0\quad (5b).$
\end{itemize}
What we want to show is that 

\[\tilde{\omega}(C)=(\alpha-4\beta)\nu_\infty+\beta\nu_f-\sum\nu_i\tilde{\mu}_i>0.\]
Combining the inequalities coming from the periods and the intersection numbers that we get

\begin{align*}
    \alpha\nu_\infty&\overset{(5b)}\geq \alpha\nu_0\Rightarrow\\
    \alpha\nu_\infty &\geq \alpha\sum\nu_i\geq \sum \tilde{\mu}_i\sum\nu_i\overset{(*)}\geq \sum\tilde{\mu}_i\nu_i\Rightarrow\\
    (\alpha-4\beta+4\beta)\nu_\infty&\geq\sum\nu_i\tilde{\mu_i}\Leftrightarrow\\
    (\alpha-4\beta)\nu_\infty+4\beta\nu_\infty&\geq \sum\nu_i\tilde{\mu}_i\overset{(2b)}{\Rightarrow}\\
    (\alpha-4\beta)\nu_\infty+\beta\nu_f&\geq\sum\nu_i\tilde{\mu}_i\Leftrightarrow\\
    \tilde{\omega}(C)&\geq 0
\end{align*}

The case $\tilde{\omega}(C)=0$ can be easily ruled out because it would imply that all inequalities are actually equalities. In particular, $(*)$ being an equality implies that $\nu_i=0$ for $i=0,1,2,3$. In turn, this implies that $\nu_\infty=0=\nu_f$ since they are already non-negative. Therefore $C$ is in the trivial homology class which is absurd. 
\end{proof}

\begin{remark}
As a reality check, it is straightforward to see that the above inequalities imply that every exceptional class has positive area.
\end{remark}

Having established what $\mathcal{K}(\widetilde{X})$ is, we next show that when the periods satisfy the desired inequality, indeed there exist two disjoint Lagrangian projective planes.

\begin{proposition}\emph{(Sufficiency)}
Consider $(X_3,\omega)$ and assume that there exists Lagrangian $\mathbb{R}P^2$'s in the homology classes $H$ and $E_1+E_2+E_3$. If, additionally, the inequality
\[\mu_1+\mu_2+\mu_3<h\]
holds, then there exist two (possibly different) Lagrangian $\R P^2$'s, $L_1$ and $L_2$, that are disjoint.
\end{proposition}

\begin{proof}
Theorem \ref{thrmss} implies the triangle inequality for the periods $\mu_i$ of $\omega$. Therefore, there exists some $\epsilon_2,\tilde{\mu_i}>0$ such that the $\tilde{\mu_i}$ satisfy

\[ \tilde{\mu}_0=\frac{\mu_1+\mu_2+\mu_3}{2}+\epsilon_2,\quad \tilde{\mu}_k=\frac{\mu_i+\mu_j-\mu_k}{2}-\epsilon_2,\]

In addition, the inequality $h-\mu_1-\mu_2-\mu_3>0$ implies, after possibly choosing an even smaller $\epsilon_2$, that there exists some $0<\epsilon_1$ such that

\[\frac{h-\sum\mu_i}{2}-\epsilon_2-\epsilon_1>0.\]
Now, let $\alpha$ and $\beta$ be defined by

\[\alpha=2h,\quad \alpha-4\beta=4\epsilon_1.\]
We can choose $\epsilon_1,\epsilon_2$ so the numbers $\alpha,\beta,\tilde{\mu}_i$ are the periods of a K\"ahler form $\tilde{\omega}$ on $\widetilde{X}$. Let us check that indeed all the inequalities of Lemma \ref{kconex5} can be satisfied by sufficiently small $\epsilon_1,\epsilon_2$.

\begin{itemize}
    \item $2\alpha\beta-4\beta^2-\sum^4_{i=0} \tilde{\mu}^2_i=h^2-\sum \mu_i^2 -4(\epsilon_1^2+\epsilon_2^2)$ and since the volume constraint is satisfied by $\omega$ it will be satisfied also for $\tilde{\omega}$, possibly even smaller $\epsilon_i$ .
    \item $\tilde{\mu_i}>0$ from the triangle inequality for the $\mu_i$.
    \item $\alpha=2h>0$.
    \item $\beta-\tilde{\mu}_0=\frac{h}{2}-\epsilon_1-\tilde{\mu}_0=\frac{h-\sum\mu_i}{2}-\epsilon_1-\epsilon_2$ where we used that $\tilde{\mu}_0=\frac{\sum \mu_i}{2}+\epsilon_2$. Since $h-\sum\mu_i>0$ by assumption, we can find $\epsilon_1,\epsilon_2$ small enough.
    \item the positivity on the $(-4)$-spheres holds by the definition since they will have area $4\epsilon_1$ and $4\epsilon_2$.
\end{itemize}
This implies that we can find a K\"ahler form on $\widetilde{X}$ with the above periods $\alpha,\beta,\tilde{\mu}_i$. Blowing down the exceptional spheres in the classes $Z_\infty$ and $\widetilde{E}_0-\widetilde{E}_1-\widetilde{E}_2-\widetilde{E}_3$ we get a symplectic form $\omega'$ with the same periods as $\omega$ but which also carries two Lagrangian $\mathbb{R}P^2$'s. Since $\omega'$ and $\omega$ have the same periods they are cohomologous and therefore $\omega$ also carries two disjoint Lagrangian $\mathbb{R}P^2$'s. 
\end{proof}

This concludes both the necessity and sufficiency part of

\begin{Theorem}\label{tworp2thrm}
Consider $(X_3,\omega)$ and suppose that the $\mathbb{Z}_2$ classes $H$ and $E_1+E_2+E_3$ carry Lagrangian $\mathbb{R}P^2$'s. Then, these Lagrangians can be made disjoint via a Hamiltonian isotopy if and only if $\sum\mu_i<h$.  
\end{Theorem}

\begin{corollary}\label{dp3cor}
In $\D_3$, both classes $H$ and $E_1+E_2+E_3$ carry a Lagrangian $\R P^2$ but these Lagrangians cannot be made disjoint by a Hamiltonian isotopy.
\end{corollary}

\subsection{In $\D_k$.}

We can generalize Corollary \ref{dp3cor} to the symplectic del Pezzo surfaces $\D_k$ for $k=4,5,6$. This is done by proving that, under the obvious homological restrictions, any exceptional sphere can be made disjoint from two disjoint Lagrangian $\R P^2$'s and thus showing that if $\D_k$ carried disjoint Lagrangian $\R P^2$'s, then so would $\D_{k-1}$ and thus deducing the desired result from $\D_3$.\par

\begin{lemma}\label{inductionstep}
Let $(X_k,\omega)$ be a rational symplectic manifold for $k\leq 6$ that has two disjoint Lagrangian $\R P^2$'s, $L_1$ and $L_2$. If there exists an exceptional symplectic sphere $E'$ such that $[L_1]\cdot [E']=0=[L_2]\cdot [E']$ then there exists an exceptional symplectic sphere $E$ homologous to $E'$ such that $E$ is disjoint from the $L_i$.
\end{lemma}

\begin{proof}
For an $X_{k}$ with $k\leq 8$, the homology classes $Y$ that carry an exceptional sphere are exactly those satisfying $Y^2=-1$ and $c_1(Y)=1$. Since these two conditions are not affected by the symplectic rational blow-up, such class $Y$ in $X_k$, for $k\leq 6$, remains exceptional after the symplectic rational blow-up. In particular, $[E']$ will still be an exceptional class in the rational manifold $(X_{k+2},\tilde{\omega})$ obtained by blowing-up $L_1$ and $L_2$.\par 

Let $S_1,S_2$ be the $(-4)$-spheres in $X_{k+2}$, replacing the $L_i$ after rational the blow-up. Theorem 1.2.7 in \cite{McDOp} shows that there exists a compatible almost complex structure $J$ such that both $S_i$ are $J$-complex and there exists a $J$-complex exceptional sphere $E$ in the homology class $[E']$. By positivity of intersections, $E$ is disjoint from the $S_i$ and thus one can see $E$ also in $X_k$, in the complement of the $L_i$. Therefore $E$ is the desired exceptional sphere.
\end{proof}

We can now prove the claimed rigidity result for del Pezzo surfaces.

\begin{Theorem}\label{delpezzothrm}
Let $L_1$ and $L_2$ be two Lagrangian $\R P^2$ in a del Pezzo surface $\D_k$ for $3\leq k \leq 6$. Then $L_1$ and $L_2$ must intersect at least at one point.
\end{Theorem}

\begin{proof}

Suppose on the contrary that $L_1$ and $L_2$ are disjoint. By Lemma \ref{existence2rp2} we can assume that $[L_1]=H$ and $[L_2]=E_1+E_2+E_3$. In $\D_3$, the statement follows from Theorem \ref{tworp2thrm}. Now, we consider $\D_n$ for $4\leq n \leq 6$. Then, there exists an exceptional sphere $S$ such that  $[S]\cdot [L_i]=0$. For example, one may take $[S]=E_4$. By Lemma \ref{inductionstep}, if $L_1$ and $L_2$ are disjoint in $\D_n$, there is an exceptional sphere in $[S]$ disjoint from the $L_i$ and then, by blowing it down, there exist disjoint Lagrangian $\R P^2$'s also in $\D_{n-1}$. Repeating this argument, if there exist disjoint Lagrangian $\R P^2$'s in $\D_n$ for $n\leq 6$, representing the classes $H$ and $E_1+E_2+E_3$, there must also exist disjoint Lagrangians in $\D_3$ which is not possible.

\end{proof}

For $k=7,8$ the situation is not very clear. First, the arguments for Lemma \ref{inductionstep} should become more nuanced, but we believe this is only a technicality. Second, in $\D_7$ there is another possible pair of disjoint Lagrangians as Lemma \ref{existence2rp2} shows and therefore this case must be separately examined. Although we don't expect that this extra case is very different from the one considered here, it is computationally more complicated so we defer its study to future work.

\section{An $L_{3,1}$ pinwheel in $S^2\times S^2$.}

In this section we will use the rational blow-up along an $L_{3,1}$ pinwheel to answer Problem $J.7$ in Evan's beautiful book \cite{evans_2023}. Following notation of Example 4 in section \ref{basicsympgeometry}, the problem can be formulated as follows:

\begin{prob}
For what values of $\lambda$ does $(S^2\times S^2,\omega_\lambda)$ support a Lagrangian $L_{3,1}$ pinwheel in the homology class $A+B$?
\end{prob}

From now on, we will always assume that the pinwheel in question will exist in the $\Z_3$-homology class $A+B$.\par 
The above problem was motivated by the fact that one can draw a visible Lagrangian $L_{3,1}$ for some values of $\lambda$.  More concretely, Evans showed

\begin{lemma}\emph{(Section $J.7$ in \cite{evans_2023})}\label{l31suf}
The symplectic manifold $(S^2\times S^2,\omega_\lambda)$ carries a visible $L_{3,1}$ pinwheel for the values $\frac{1}{2}<\lambda< 2$. 
\end{lemma}

We will show that the inequality $\frac{1}{2}<\lambda<2$ is sharp. In fact, it will be more natural to work with the symplectic form $\omega_{a,b}$ on $S^2\times S^2$, where $a,b>0$ are the areas of the two product factors. With this notation, clearly $\omega_\lambda=\omega_{\lambda,1}$. We note that the proof below is structurally exactly the same as the proof of Proposition \ref{2rpnec} so we supress some of the details.

\begin{proposition}\label{l31nec}
Let $(S^2\times S^2,\omega_{a,b})$ carry a Lagrangian $L_{3,1}$ in the homology class $A+B$. Then
\[\frac{b}{2}<a<2b.\]
\end{proposition}

\begin{proof}
Let $L$ be a Lagrangian $L_{3,1}$ pinwheel in $A+B$. We begin with the observation that the subspace of elements of $H_2(S^2\times S^2,\Z)$ that pair trivially with $L$ is spanned by  $A+2B$ and $B+2A$.\par

Rationally blowing-up $L$, we replace it with two symplectic spheres, $S_1$ and $S_2$, of self intersection $(-5)$ and $(-2)$ respectively (see Section \ref{ratblowupsec}). The two spheres intersect at one point positively.\par 

By Theorem \ref{ratblowuponratman}, the blown-up symplectic manifold $(\widetilde{X},\tilde{\omega})$ is again a rational symplectic manifold. Since we introduce two new spheres we have that $b_2(\widetilde{X})=b_2(S^2\times S^2)+2$ which means that we can identify the blown-up space with $(X_3,\tilde{\omega})$, for some symplectic form $\tilde{\omega}$. We will denote the standard periods of $\tilde{\omega}$ with $h,\mu_1,\mu_2,\mu_3$.\par

A lengthy calculation shows that, up to diffeomorphism, the classes that can support $(-5)$ symplectic spheres are $-H-2E_1+E_2+E_3,-2H+3E_1$ and $-E_1-E_2$. Similarly, the classes that can support a $(-2)$ symplectic sphere are $E_1-E_2$ and $H-E_1-E_2-E_3$. 
\par
Recall that the diffeomorphism between $S^2\times S^2-\nu L$ and $X_3-\nu \mathcal{S}_3$ implies the existence of a lattice isomorphism between classes in $H_2(S^2\times S^2,Z)$ which pair trivially $\mod 3$ with $L$ and classes in $H_2(X_3,Z)$ that pair trivially with $\langle S_1, S_2\rangle$. This isomorphism dictates that $S_1$ is in the class $-2H+3E_1$ and $S_2$ is in the class $H-E_1-E_2-E_3$.\par

Since the spheres $S_1$ and $S_2$ are symplectic, there exists positive real numbers $\epsilon_1$ and $\epsilon_2$ such that

\[
    -2h+3\mu_1\coloneqq 9\epsilon_1>0,\quad h-\mu_1-\mu_2-\mu_3\coloneqq 9\epsilon_2>0
\]

In addition, the $\omega_{a,b}$-area of the classes $A+2B,B+2A$ should be the same as the $\tilde{\omega}$-area of $3H-2E_1-E_2$ and $3H-2E_1-E_3$ and therefore we get the following system of equations:

\begin{align*}
    a+2b&=3h-2\mu_1-\mu_2 \\ 
     2a+b&=3h-2\mu_1-\mu_3\\
    9\epsilon_1&=-2h+3\mu_1\\
    9\epsilon_2&=h-\mu_1-\mu_2-\mu_3
\end{align*}
Inverting this system we get

\begin{align*}\label{solved}
    h&=a + b + 3\epsilon_1 - 3\epsilon_2\\
    \mu_1&=\frac{2(a+b)}{3}+5\epsilon_1-2\epsilon_2\\
    \mu_2&=\frac{2a-b}{3}-(\epsilon_1+5\epsilon_2)\\
    \mu_3&=\frac{2b-a}{3}-(\epsilon_1+5\epsilon_2)
\end{align*}
Since $\tilde{\omega}$ is a symplectic form, it satisfies the inequalities $\mu_i>0$ and $h>\mu_i+\mu_j$. In particular, $\mu_2>0$ and $\mu_3>0$ imply that $2a>b$ and $2b>a$. $\square$
\end{proof}

Putting everything together, we can prove

\begin{Theorem}\label{jquestion}
In $(S^2\times S^2,\omega_\lambda)$, the $\Z_3$-homology class $[A+B]$ carries a Lagrangian $L_{3,1}$ pinwheel if and only if

\[\frac{1}{2}<\lambda<2.\]

\end{Theorem}

\begin{proof}
On one hand, the necessity of the inequality $\frac{1}{2}<\lambda<2$ follows from Proposition \ref{l31nec}, setting for $a=\lambda$ and $b=1$. On the other hand, Lemma \ref{l31suf} shows that if the inequality holds, then the $L_{3,1}$ pinwheel indeed exists.
\end{proof}

\printbibliography

@article{SS,
    author = "V. Shevchishin and G. Smirnov",
    title = "{Symplectic triangle inequality}",
    journal = " Proc. Amer. Math. Soc.",
    volume = "148(4)",
    pages = "1389–1397",
    year = "2020",
}

@article{BLW,
    author = "M. S. Borman and T.-J. Li and W. Wu",
    title = "{Spherical Lagrangians via ball packings and symplectic cutting}",
    journal = "Sel. Math. New Ser.",
    volume = "20",
    pages = "261--283",
    year = "2014",
}

@article{AudinLag,
    author = "M. Audin",
    title = "{Fibrés normaux d’immersions en dimension double, points dou-
bles d’immersions lagragiennes et plongements totalement réels}",
    journal = "Comment. Math. Helv.",
    volume = "63(4)",
    pages = "593--623",
    year = "1988",
}

@article{LLWuniq,
    author = "J. Li and T.J. Li and W. Wu",
    title = "{The Symplectic Mapping Class Group of $\mathbb{C}P^2 \#n\mathbb{C}P^2$
with $n\leq 4$}",
    journal = "Michigan Math. J.",
    volume = "64",
    pages = "319--333",
    year = "2015",
}

@article{EKb,
    author = "J. Evans",
    title = "{A Lagrangian Klein bottle you can’t squeeze}",
    journal = "J. Fixed Point Theory Appl.",
    volume = "24(47)",
    year = "2022",
}

@article{McDOp,
    author = "D. McDuff and E. Opshtein",
    title = "{Nongeneric $J$–holomorphic curves and singular inflation}",
    journal = "Algebr. Geom. Topol.",
    volume = "15(1)",
    pages = " 231--286",
    year = "2015",
}

@article{BS,
    author = "J. Brendel and F. Schlenk",
    title = "{Pinwheels as Lagrangian barriers}",
    journal = "Commun. Contemp. Math.",
    volume = "26(5)",
    year = "2024",
}

@misc{khodorovskiy2013symplectic,
      title={Symplectic Rational Blow-up}, 
      author={T. Khodorovskiy},
      year={2013},
      eprint={1303.2581},
      archivePrefix={arXiv},
      primaryClass={math.SG}
}

@book{evans_2023, place={Cambridge}, series={London Mathematical Society Student Texts}, title={Lectures on Lagrangian Torus Fibrations}, DOI={10.1017/9781009372671}, publisher={Cambridge University Press}, author={J. Evans}, year={2023}, collection={London Mathematical Society Student Texts}}

@article{Sym,
    author = "M. Symington",
    title = "{Generalized symplectic rational blowdowns}",
    journal = "Algebr. Geom. Topol.",
    volume = "1",
    pages = "503-–518",
    year = "2001",
}

@misc{li2008space,
      title={The Space of Symplectic Structures on Closed 4-Manifolds}, 
      author={T. J. Li},
      year={2008},
      eprint={0805.2931},
      archivePrefix={arXiv},
      primaryClass={math.SG}
}

@article{ParkShin,
    author = "H. Park and D. Shin",
    title = "{Symplectic rational blow-ups on rational 4-manifolds}",
    journal = "Proc. Amer. Math. Soc.",
    DOI={https://doi.org/10.1090/proc/16519},
    year = "2023",
}

@article{Dorfm,
    author = "J. Dorfmeister",
    title = "{ Minimality of symplectic fiber sums along spheres}",
    journal = "Asian J. Math.",
    volume = "17(3)",
    pages = "423-–442",
    year = "2013",
}

@article{DHL,
    author = "B. Dai and C.I. Ho and T.J. Li",
    title = "{Nonorientable Lagrangian surfaces in rational 4–manifolds}",
    journal = "Algebr. Geom. Topol.",
    volume = "19(6)",
    pages = "2837–-2854",
    year = "2019",
}

@article{ES,
    author = "J. Evans and I. Smith",
    title = "{Markov numbers and Lagrangian cell complexes in the complex projective plane}",
    journal = "Geom. Topol.",
    volume = "22(2)",
    pages = "1143--1180",
    year = "2018",
}
\end{document}